\documentclass[11pt,reqno]{amsart}

\makeatletter
\usepackage{amssymb}
\usepackage{amssymb,amsmath,color,comment}
\usepackage{amsbsy}
\usepackage{amsfonts}
\usepackage[notref,notcite]{showkeys}

\def\marginpar#1{\ignorespaces}

\newtheorem{thm}{Theorem}

\newtheorem{lem}[thm]{Lemma}
\newtheorem{cor}[thm]{Corollary}
\newtheorem{hyp}{Assumption}

\theoremstyle{definition}

\newtheorem{rmq}[thm]{Remark}
\newtheorem{ex}{Example}

\def\eqlabel#1{\def\@currentlabel{#1}}

\def\formula#1{\def\@tempa{#1}\let\@tempb\theequation\def\theequation{%
\hbox{#1}}\def\@currentlabel{(\theequation)}$$}
\def\endformula{\leqno\hbox{(\@tempa)}$$\@ignoretrue\let\theequation\@tempb}

\def\given{\hskip5\p@\relax\vrule\@width.4\p@\hskip5\p@\relax}

\newcommand{\open}[1]{%
\par\normalfont\topsep6\p@\@plus6\p@\trivlist\item[\hskip\labelsep\itshape#1%
\@addpunct{.}]\ignorespaces}

\DeclareRobustCommand{\close}[1]{%
  \ifmmode % if math mode, assume display: omit penalty etc.
  \else \leavevmode\unskip\penalty9999 \hbox{}\nobreak\hfill
  \fi
  \quad\hbox{$#1$}}

\newlength{\toskip}

\def \po {\left(}
\def \pf {\right)}
\def \co {\left[}
\def \cf {\right]}

\def \R {\mathbb R}
\def \dd {\text{d}}
\def\na {\nabla}

\makeatother

\begin{document}
%%%%%%%%%%%%%%%%%%%%%%%%%%%%%%%%%%%%%%%%%%%%%%%%%%%%%%%%%%%%%%%%%%%%%%%
%%%%%%%%%%%%%%%%%%%%%%%%%%%%%%%%%%%%%%%%%%%%%%%%%%%%%%%%%%%%%%%%%%%%%%%
\date{\today}

\title[Entropic convergence for Langevin diffusion]{Entropic multipliers method for Langevin diffusion and weighted log Sobolev inequalities.}

\author[P. Cattiaux]{\textbf{\quad {Patrick} Cattiaux $^{\spadesuit}$ \, \,  }}
\address{{\bf {Patrick} CATTIAUX}\\ Institut de Math\'ematiques de Toulouse, Universit\'e de Toulouse, CNRS UMR 5219, \\ 118 route de Narbonne, F-31062 Toulouse cedex 09.}
\email{patrick.cattiaux@math.univ-toulouse.fr}

\author[A. Guillin]{\textbf{\quad {Arnaud} Guillin $^{\diamondsuit}$ \, \, }}
\address{{\bf {Arnaud} GUILLIN}\\ Laboratoire de Math\'ematiques, CNRS UMR 6620, Universit\'e Blaise Pascal,
avenue des Landais, F-63177 Aubi\`ere.} \email{guillin@math.univ-bpclermont.fr}

\author[P. Monmarch\'e]{\textbf{\quad {Pierre} Monmarch\'e $^{\clubsuit}$ \,  }}
\address{{\bf {Pierre} MONMARCHE}\\  CERMICS (ENPC), CNRS UMR 8050\\	8 Avenue Blaise Pascal, 
77455 Marne-la-Vall\'ee.} \email{pierre.monmarche@ens-cachan.org}

\author[C. Zhang]{\textbf{\quad {Chaoen} Zhang $^{\diamondsuit}$ \, \, }}
\address{{\bf {Chaoen} ZHANG}\\ Laboratoire de Math\'ematiques, CNRS UMR 6620, Universit\'e Blaise Pascal,
avenue des Landais, F-63177 Aubi\`ere.} \email{zhang@math.univ-bpclermont.fr}

\maketitle

 \begin{center}

 \textsc{$^{\spadesuit}$  Universit\'e de Toulouse}
\smallskip

\textsc{$^{\diamondsuit}$ Universit\'e Blaise Pascal}
\smallskip

\textsc{$^{\clubsuit}$ Universit\'e Paris-Est, INRIA Paris}
\smallskip

\end{center}

\begin{abstract}
In his work about hypocercivity, Villani \cite{Villani} considers in particular convergence to equilibrium for the kinetic Langevin process. While his convergence results in $L^2$ are given in a quite general setting, convergence in entropy requires some boundedness condition on the Hessian of the Hamiltonian. We will show here how to get rid of this assumption in the study of the hypocoercive entropic relaxation to equilibrium for the Langevin diffusion. Our method relies on a generalization to entropy of the multipliers method and an adequate functional inequality. As a byproduct, we also give tractable conditions for this functional inequality, which is a particular instance of a weighted logarithmic Sobolev inequality, to hold.
\end{abstract}

\bigskip

\textit{ Key words : }Langevin diffusion, entropic convergence, hypocercivity, weighted logarithmic Sobolev inequality, Lyapunov conditions.
\bigskip

%\textit{ MSC 2010 : } 
%\bigskip

\section{Settings and main results.}

Let $U: \R^{d}\rightarrow \R$ be a smooth function such that $U\geq 1$ and $\int e^{-U(x)}dx$ is finite. $U$ will represent the confinement potential for the Hamiltonian $H(x,y)=U(x)+\frac{1}{2}|y|^2$ defined on $\R^{2d}$. The associated Boltzmann-Gibbs (probability) measure is given by
$$
\dd\mu  =  \frac{1}{Z} \; e^{-H(x,y)}\dd x \dd y
$$
where $Z$ is the normalizing constant $\int e^{-H(x,y)}\dd x \dd y$.\\ The Langevin dynamics associated to this measure is a flow of probability measures $\dd \mu_t = f_t \, \dd\mu$ for $t\geq 0$, where $f_t$ solves (at least in a weak sense) the Langevin equation  $$\partial_t f_t = L f_t \, ,$$ $L$ being given by
\begin{eqnarray}\label{EqGeneLangevin}
L   &=& -y.\na_x  +\po \na U(x)-y\pf.\na_y + \Delta_y \; .
\end{eqnarray}
We are thus interested in solutions belonging to $\mathbb L^1(\mu)$. Of course, the hypoelliptic regularity theorem ensures that $(t,x,y) \mapsto f_t(x,y)$ is smooth on $\mathbb R_+^*\otimes \R^{2d}$, whatever the regularity of $f_0$. It is then easy to see that mass and positivity are preserved so that if $f_0 \dd \mu$ is a probability measure so is $f_t \dd \mu$ for any $t\geq 0$. \\ The corresponding stochastic process is given by the S.D.E. $$
\left\{\begin{array}{l}
dx_t=y_tdt\\
dy_t=-y_tdt-\nabla U(x_t)dt+\sqrt{2}dW_t
\end{array}\right.
$$ where $(W_t)$ is an usual $d-$dimensional Wiener process. The infinitesimal generator of the process is thus $L^*= y.\na_x  -\po \na U(x)+y\pf.\na_y + \Delta_y$. The law $\mu$ is the unique invariant (but not reversible) probability measure for the process, and $\dd\mu_t=f_t \dd \mu$ is the distribution of the process at time $t$. One can also write down the P.D.E. satisfied by $\mu_t$ (or its density w.r.t. Lebesgue measure) which is usually called the kinetic Fokker-Planck equation. We denote by $P_t=e^{tL}$ the semi-group on $\mathbb L^1(\mu)$ with generator $(L,D(L))$, i.e. $f_t=P_tf_0$. 
\medskip

We are interested in the long time behavior of the Langevin diffusion. The usual ergodic theorem tells us that $\frac 1t \, \int_0^t \, \mu_s \, \dd s$ weakly converges to $\mu$ as $t$ grows to infinity. One can thus ask for the convergence of $f_t$ towards $1$ as $t$ goes to infinity. \\ This question has been investigated by many authors in recent years both in the PDE community and the probability community. One of the main difference is of course the way to look at this convergence:  total variation distance, $\mathbb L^2(\mu)$ norm, $\mathbb H^1(\mu)$ semi-norm, relative entropy, Wasserstein distance. Another associated problem is to get some bounds on the rate of convergence, once convergence holds true. Let's review some results in this direction.
\medskip

More or less at the same time, both probabilists and PDE specialists have considered the problem of the speed of convergence to equilibrium. Talay \cite{Tal02} and Wu \cite{Wu01} have built Lyapunov functions and using Meyn-Tweedie's approach have established (non quantitative) exponential convergence to equilibrium (see also \cite{BCG08} for this approach for kinetic models) under quite general assumptions. Desvillettes and Villani \cite{DV01} used an heavy Fourier machinery to established sub-exponential entropic convergence. Then H\'erau and Nier \cite{HN04} have carried out the spectral analysis of this equation and thus obtained a $\mathbb L^2$ exponential decay with quite sharp constants under general conditions. It has settled the bases for the theory of hypocercivity of Villani \cite{Villani} for the $\mathbb L^2$ and the entropic convergence to equilibrium, when $\text{Hess}(U)$ is bounded in the entropic case, see also \cite{MMS15} for a version without regularity issues. Finally, and quite recently, coupling approaches, using synchronous coupling or coupling by reflection (see \cite{BGM10} or \cite{EGZ16,EGZ17}) have established exponential convergence to equilibrium in Wasserstein distance with sharp constants, once again when $\text{Hess}(U)$ is bounded.

\bigskip

 \noindent As we will adopt the terminology and adapt the methodology of hypocoercivity as in Villani \cite{Villani}, let us describe a little bit further the formalism of this setting. Recall that the variance of a squared integrable function $g$ with respect to $\mu$ is defined by $$\text{Var}_\mu(g):=\int g^2\dd\mu-\left(\int g\dd\mu\right)^2=\int \left(g-\int g\dd\mu\right)^2\dd\mu$$ while the entropy is defined for  positive functions by $$\text{Ent}_\mu(f):=\int f\ln f\dd\mu-\int f\dd\mu \ln \int f\dd\mu \; .$$ The law $\mu$ is said to satisfy a Poincar\'{e} inequality if there exists a positive constant $C_P$ such that for all smooth functions $g$
\[
\text{Var}_\mu(g) \leq C_P \int |\nabla g|^2 \dd\mu \, .
\]
Similarly, $\mu$ satisfies a logarithmic Sobolev (or log-Sobolev in short) inequality if there exists a constant $C_{LS}$ such that for all smooth functions $g$, 
\[
\text{Ent}_\mu(g^2) \leq C_{LS} \int |\nabla g|^2 \dd\mu \, .
\]
The natural $\mathbb H^1_\mu$ semi-norm is defined as $||g||_{H^1_\mu}:=||\nabla g||_{\mathbb L^2_\mu}$. Exponential convergence of $P_tf_0$ to $1$ in $\mathbb H^1_\mu$ and variance was proved by Villani \cite{Villani} under two conditions: 
\begin{enumerate}
\item[(1-var)]
\quad $ |\nabla^2 U|\leq c \; (1+|\nabla U|)$;
\item[(2-var)] \quad $e^{-U(x)}\dd x$ satisfies a Poincar\'{e} inequality.
\end{enumerate}
Remark that (2-var) is equivalent to the fact that $\mu$ satisfies a Poincar\'e inequality, thanks to the tensorization property of the latter, since the gaussian measure satisfies a Poincar\'e inequality.
\medskip

\noindent For convergence in entropy, the assumptions made by Villani are much stronger: 
\begin{enumerate}
\item[(1-ent)] \quad $\nabla^2 U$ is bounded; \item[(2-ent)] \quad  $e^{-U(x)} \, \dd x$ satisfies a log -Sobolev inequality.
\end{enumerate}
Again, (2-ent) is equivalent to the fact that $\mu$ satisfies a log-Sobolev inequality, thanks to a similar argument of tensorization.\\ When both these assumptions are satisfied, Villani showed that, for any initial probability density $f_0$ with finite moments of order 2, the entropy of $P_tf_0$ converges to $0$ exponentially fast (see Villani \cite{Villani} Theorem 39). 
\medskip

\noindent Our main goal in this paper is to get rid of the boundedness assumption (1-ent) for $\nabla^2 U$, replacing it by 
\begin{hyp}\label{HypB} 
there exists $\eta \geq 0$ such that $U^{-2\eta}\na^2 U$ is bounded.
\end{hyp}
\noindent A typical situation where Assumption \ref{HypB} is satisfied is when both $U$ and $\nabla^2 U$ have polynomial growth at infinity, i.e. $U(x)\geq c_1 \, (1+|x|)^l$ and $|\na^2U|\leq c_2 \, (1+|x|)^j$ so that we may choose $\eta \geq \frac{j}{2l}$. In particular if $j=l-2\geq 0$ as it is the case for true polynomials of degree at least 2, we may choose $\eta = \frac12-\frac{1}{l}$.
\medskip

\noindent The counterpart is that we have to reinforce (2-ent) replacing it by the stronger
\begin{hyp}\label{HypU}
$\mu$ satisfies the following weighted log-Sobolev inequality: there exists $\rho>0$ s.t. for all smooth enough $g$ with $\int g^2 \, \dd \mu=1$:
\begin{equation}\label{eqlspoids}
\mbox{\rm Ent}_\mu(g^2) \leq  \rho \int  (H^{-2\eta}|\na_x g|^2 + |\na_y g|^2)\dd \mu.
\end{equation}
\end{hyp}

\noindent Once both Assumptions \ref{HypB} and \ref{HypU} are satisfied, we can prove exponential decay in entropy for the Langevin diffusion. Our approach is based on the multiplier method. More precisely we will prove the following:
\begin{thm}\label{ThmHypocoPoids}
Under Assumptions \ref{HypB} and \ref{HypU}, let
\begin{eqnarray*}
\lambda & = & \po \| H^{-2\eta}\na^2 U\|_{\infty} + 2\pf^2\, ,\\
\kappa & = & \frac1{1300\po  \eta+d\pf^4 } \, .
\end{eqnarray*}
Then for all initial probability density $f$,
$$
\mbox{\rm Ent}_\mu(P_t f) \leq  \exp\po{-\frac{\kappa}{1+4\lambda\rho}\int_0^t (1-e^{-s})^2\dd s}\pf \mbox{\rm Ent}_\mu(f) \, .
$$
\end{thm}

\noindent Section 2 is devoted to the proof of this theorem which contains Villani's result in the case $\eta=0$. The key idea is to use a twisted gradient depending on time, see lemma \ref{LemGammaPoids}. An important aspect of our result is that the bounded Hessian condition in Villani's approach is relaxed as Assumption 1. In fact it was a major issue raised by Villani \cite{Villani} concerning the entropic convergence.  Indeed, his $L^2$ multiplier method, at the basis of the entropic hypocercivity, does  not rely on a Poincar\'e inequality but on Brascamp-Lieb inequality. It was thus thought that for the multiplier method to hold for entropy, an entropic Brascamp-Lieb inequality was needed. However Bobkov-Ledoux \cite{BL00} proved that this inequality is false in general, and true in very particular setting. Our strategy is then to show that it is not an entropic Brascamp-Lieb inequality that we need but a particular weighted logarithmic Sobolev inequality. Note also that a first attempt to skip the boundedness assumption for the Hessian is contained in \cite{BCG08} Theorem 6.10, but the statement therein is much weaker than the one of the present theorem and most importantly not at all quantitative .
\medskip

Next we shall show that, similarly to the non weighted case studied in \cite{CG16} (see also \cite{BBCG,CGWW09}), the weighted log Sobolev inequality in Assumption \ref{HypU} is equivalent to some Lyapunov type condition. \\ To this end we introduce the natural second order operator $$L_{\eta}:=H^{-2\eta} \Delta_x +  \Delta_y - H^{-2\eta}\left(2\eta\frac{\nabla_x H}{H}+\nabla_x H\right).\nabla_x - \nabla_y H.\nabla_y \, ,$$ which is symmetric in $\mathbb L^2_\mu$ and satisfies
\begin{equation}\label{eqIPP}
\int \, f \, L_\eta g \, \dd \mu = - \, \int \, (H^{-2\eta} \, \nabla_x f.\nabla_x g+ \nabla_y f.\nabla_y g) \, \dd \mu \, .
\end{equation}
\begin{thm}\label{Thm-LyapunovCondition}
Assume that $U$ goes to infinity at infinity, that $|\nabla H|\ge h>0$ outside some large ball. Denote $A_r:=\{(x,y): H(x,y)\leq r\}$, and
\[
\theta(r)=\sup\limits_{z\in \partial A_r} \max\limits_{i,j=1,...,2d} |\frac{\partial^2 H}{\partial z_i\partial z_j} |
\] 

Assume that $\theta(r)\leq ce^{C_0r}$ with some positive constants $C_0$ and $c$  for $r$ sufficiently large. Assume that there exists a Lyapunov function $W$ with $W(x)\ge w>0$ for all $(x,y)$ and some $\lambda,b>0$ satisfying
$$L_{\eta} W(x,y)\le -\lambda H(x,y)\, W(x,y)+b\, . $$
Then $\mu$ verifies a weighted logarithmic Sobolev inequality \eqref{eqlspoids}.
\end{thm}

Remark that the condition $\theta(r)\le ce^{{C_0}r}$ is trivially verified when both $U$ and $\mbox{\rm Hess}(U)$ have a polynomial growth. Also, a Lyapunov function exists if $U$ satisfies the conditions in the following corollary:
\begin{cor}\label{Cor-Lyapunov}Assume that the following conditions hold outside a compact domain
\begin{enumerate}
  \item $\Delta_x U\leq \kappa |\nabla_x U|^2$  for some $\kappa\in (0,1)$;
  \item a growth condition: $|\nabla_x U|^2 \geq c U^{2\eta+1}$ for some positive constant $c$.
\end{enumerate}
Then $d\mu=\frac{1}{Z}e^{-H(x,y)}dxdy$ satisfies a weighted logarithmic Sobolev inequality.

Moreover, if we assume that $U^{-2\eta}\nabla^2 U$ is bounded, then we may apply Theorem \ref{ThmHypocoPoids}.
\end{cor}

The next section will present the proof of Theorem 1, where the entropic multipliers method is presented. In Section 3, the treatment via Lyapunov condition of weigthed log-Sobolev inequality, i.e. Theorem 2 and Corollary 3, is done.
The final section discusses some additional points on weighted inequalities. Indeed, the proof of weighted Poincar\'e inequality used by Villani relies solely on some Poincar\'e inequality for each measure and adapt the usual argument of tensorization using heavily the orthogonality inherited from the $\mathbb L^2_\mu$ structure. However, in the entropic case, {from a log-Sobolev for each marginal, we are only able to recover a weaker inequality for the product measure.}
\bigskip

%%%%%%%%%%%%%%%%%%%%%%%%%
%%%%%%%%%%%%%%%%%%%%%%%%%

\section{Proof of Theorem \ref{ThmHypocoPoids}.}
This section is devoted to the proof of Theorem \ref{ThmHypocoPoids}. Actually we will prove a more general statement. Consider an admissible function $\Psi$, that is $\Psi\in C^4$ and $\frac{1}{\Psi''}$ is positive concave, as in \cite{Mon15b}. \\ Theorem \ref{ThmHypocoPoids} corresponds to $$\Psi: \R^+\rightarrow \R, u \mapsto u\ln u+1-u \, ,$$ while the $\mathbb L^2_\mu$ case corresponds to $\Psi(u)=(u-1)^2$. We also define $\psi= \Psi''$. 

{We only consider the case where $f_0$ is bounded away from zero. Indeed, if it is not the case, writing $g_0 = (1-\delta) f_0 + \delta$ for some $\delta>0$, then we may prove Theorem \ref{ThmHypocoPoids} for $g_t = (1-\delta) f_t + \delta$ and let $\delta$ go to zero to recover the result for $f_t$.} 
\medskip

In this general framework we replace the weighted log-Sobolev inequality in Assumption \ref{HypU} by the following, satisfied for any bounded density of probability $f$,
\begin{equation}\label{assump2}
\int \, \Psi(f) \, \dd \mu \leq \rho \, \int \, \psi(f) \, \left(H^{-2\eta}|\nabla_x f|^2 + |\nabla_y f|^2 \right)\dd \mu.
\end{equation}
We shall obtain the analogue of Theorem \ref{ThmHypocoPoids}, replacing the entropy by $\int \Psi(f) \dd \mu$, i.e. if \eqref{assump2} and assumption \ref{HypB} are satisfied, then for all initial probability density $f$,
\begin{equation}\label{convpsi}
\Psi(P_t f) \leq  \exp\po{-\frac{\kappa}{3+\lambda\rho}\int_0^t (1-e^{-s})^2\dd s}\pf \Psi(f) \, .
\end{equation}
\bigskip

\noindent The key point of the proof is to introduce a time {and space}-dependent twisted gradient. Consider $r\in\mathbb N$ and for $0 \leq i \leq r$, $x\mapsto b_i(x)\in\R^d$ a smooth vector field, $C_i = b_i.\na$, $Cf=(C_0 f,\dots,C_r f)$,  %$t,x\mapsto m_i(t,x) \in\R_+$ be a smooth scalar weight. We write $M=diag(m_0,\dots, m_r)$ and $Cf=(C_0 f,\dots,C_r f)$ so that
%\[(C f)^TM C f = \sum_{i=1}^r m_i| C_i f|^2,\]
% and let
 $t,x \mapsto M_t(x)$  a smooth function from $\R_+\times\R^{d}$ to $\mathcal M_{r\times r}^{sym+}(\R)$ the set of positive semi-definite symmetric real matrices of size $r$, and
\begin{eqnarray*}
F(t) &=& \int \psi(P_t f) \po C P_t f\pf^T M_t C P_t f \dd \mu
\end{eqnarray*}
where $A^T$ stands for the transpose of the matrix  $A$ and vectors are seen as 1-column matrices. The following results holds for any diffusion operator:
\begin{lem}\label{LemGammaPoids}
Let $L=L_s+L_a$, where $L_s=\frac12(L+L^*)$ and $L_a=\frac12(L-L^*)$ stand for the symmetric and antisymmetric part of $L$ in $\mathbb L^2_\mu$. Then
\begin{eqnarray*}
F'(t) &\leq& \int \psi(P_t f) \po C P_t f\pf^T \Big(2 M_t   \co C,L\cf   + \po (2L_s-L)M_t+\partial_t M_t\pf  C  \Big) P_t f \dd \mu.
\end{eqnarray*}
where $\co C_i,L\cf = C_i L-L C_i $ is the (generalized) Lie bracket of $C_i$ and $L$ and $\co C,L \cf = \po \co C_0,L\cf ,\dots,\co C_r,L\cf \pf$.
\end{lem}
\begin{proof}
In the following we write $f$ for $P_t f$ and $M_t(x)=(m_{i,j}(t,x))_{0\leq i,j\leq r}$. First it holds
\begin{eqnarray*}
\partial_t\po \int \psi(f) m_{i,j} C_i f C_j f \dd \mu \pf &=& \int \psi(f) \partial_t (m_{i,j}) C_i fC_j f  + m_{i,j} \partial_t\po \psi(f) C_i fC_j f \pf \dd \mu.
\end{eqnarray*}
This derivation is justified by the fact that $f_0$ is uniformly strictly positive and so is $f_t$ by hypoellipticity and the control of the growth of the derivative of $f_t$, using Villani \cite[Sect. A.21]{Villani} or \cite{GW12}. Denote as usual the Carr\'e-du-Champ operator $2\,  \Gamma(g,h)= L(gh)-gLh-hLg$.
Next, $\mu$ being invariant for $L$, and using the diffusion property, i.e. that the chain rule property $L\Psi(f_1,...,f_d)=\sum_1^d\partial_i\Psi(f)Lf_i+\sum_{i,j}\partial_{i,j}\Psi(f)\Gamma(f_i,f_j)$ holds for all nice $\Psi$ and $f$,
\begin{eqnarray*}
0 &=& \int L\po m_{i,j} \psi(f) C_i fC_j f \pf \dd \mu\\
&=& \int L\po m_{i,j} \pf \psi(f) C_i fC_j f  \dd \mu + \int  m_{i,j}  L\po \psi(f)C_i fC_j f  \pf \dd \mu\\
& & + 2\int \Gamma\po m_{i,j} ,  \psi(f) C_i fC_j f \pf \dd \mu\\
&=& \int (L-2L_s)\po m_{i,j} \pf \psi(f) C_i fC_j f \dd \mu + \int  m_{i,j}  L\po \psi(f) C_i fC_j f  \pf \dd \mu \, .
\end{eqnarray*}

The case where $M$ is constant (and symmetric semi-definite positive) is already treated in \cite[Lemma 8]{Mon15b} where it is shown that
\begin{eqnarray*}
 \sum_{i,j} m_{i,j} \Big(L\po \psi(f) C_i fC_j f  \pf - \partial_t\po \psi(f) C_i fC_j f \pf  \Big)&\geq& 2 \psi(f)  \sum_{i,j} m_{i,j} \po C_i f\pf  \co L,C_j\cf f \, .
\end{eqnarray*}
The proof follows by taking the integral of both sides.
\end{proof}

\begin{proof}[Proof of Theorem \ref{ThmHypocoPoids}]
Now consider the case of the Langevin diffusion, namely $L$ is given by \eqref{EqGeneLangevin}. Note that
\[\co L,\na_y \cf = \na_x + \na_y\hspace{40pt} \co L,\na_x \cf = -\na^2 U(x).\na_y. \]
The operator $L$ is decomposed as $L=L_s + L_a$ where
\[L_s = - y.\na_y + \Delta_y\hspace{40pt}L_a =- y.\na_x  + \na U(x) .\na_y  .\]
Recalling $H(x,y) = U(x)  + \frac12 |y|^2$, then $L_a H=0$ and more generally $L_a (g \circ H)=0$ for any smooth $g:\mathbb R\rightarrow \R$. In particular for $\eta >0$,
\begin{eqnarray*}
(2L_s-L)\po H^{-\eta}\pf& =& L_s\po H^{-\eta}\pf\\
& =& \eta(|y|^2+d)H^{-\eta-1} + \eta(\eta+1)|y|^2H^{-\eta-2} \\
%& &\\
& \leq & \eta(2\eta+d+4) H^{-\eta}.
\end{eqnarray*}
Let $a,b,c$ depend on $t$ and $H(x,y)$, and let $M=\begin{pmatrix}
a & b \\ b & c
\end{pmatrix}$ and $C=\na$, so that Lemma \ref{LemGammaPoids} reads
\begin{eqnarray*}
F'(t)% &\leq& \int a(P_t f) \po \na P_t f\pf^T \Big(2 M_t   \co C,L\cf   + \po (2L_s-L)M_t+\partial_t M_t\pf  \na  \Big) P_t f \dd \mu.\\
&\leq& -2\int \psi(P_t f) \po \na P_t f\pf^T N\na P_t f \dd \mu
\end{eqnarray*}
with
\begin{eqnarray*}
N &=& \begin{pmatrix}
b -\frac12 (L_s+\partial_t)a & & -  a \na^2 U+b  -\frac12  (L_s+\partial_t)b\\
& & \\
c -\frac12  (L_s+\partial_t)b & & -  b \na^2 U+c  -\frac12  (L_s+\partial_t)c
\end{pmatrix} .
\end{eqnarray*}
In the top left corner  $b$  is good news since it gives some coercivity in the $x$ variable. Nevertheless as soon as $b\neq 0$, $b\na^2U$ in the bottom right corner is an annoying term that can only be controlled by the entropy production if it is bounded (which is where, in the previous studies, the assumption that $\na^2 U$ is bounded barged in).
\medskip

\noindent Writing $\alpha(t) = (1-e^{-t})$, set
\[c=2\varepsilon\alpha H^{-\eta} \hspace{25pt}b=\varepsilon^2\alpha^2H^{-2\eta}\hspace{25pt}a=\varepsilon^3\alpha^3H^{-3\eta}\]
for some $\varepsilon\in(0,1)$. In other words,
\begin{eqnarray*}
(\na f)^T M\na f &=& \varepsilon \alpha H^{-\eta} |\na_y f|^2 + \varepsilon \alpha H^{-\eta} |\na_y f + \varepsilon\alpha H^{-\eta} \na_x f|^2,
\end{eqnarray*}
so that, in particular, $M$ is positive definite. In that case we bound
\begin{eqnarray*}
b -\frac12 (L_s+\partial_t)a
& \geq & \varepsilon^2\alpha^2 H^{-2\eta} - \frac32\eta(6\eta+d+4)\varepsilon^3\alpha^3 H^{-3\eta} - \frac32\varepsilon^3\alpha^2 e^{-t} H^{-3\eta}\\
& \geq & \varepsilon^2\alpha^2 H^{-2\eta}\po 1 -\po\frac32\eta(6\eta+d+4)\alpha  + \frac32 e^{-t}\pf \varepsilon\pf\\
& \geq & \varepsilon^2\alpha^2 H^{-2\eta}\po 1-   9(\eta+d)^2\varepsilon\pf
\end{eqnarray*}
\begin{eqnarray*}
 -  b \na^2 U+c  -\frac12  (L_s+\partial_t)c
& \geq & -\varepsilon^2 \alpha^2 \| H^{-2\eta}\na^2 U\|_{\infty} + 2\varepsilon\alpha H^{-\eta} -\eta(2\eta+d+4)\varepsilon\alpha H^{-\eta} - \varepsilon e^{-t}H^{-\eta} \\
& \geq & -\varepsilon^2 \alpha^2 \| H^{-2\eta}\na^2 U\|_{\infty} - \varepsilon H^{-\eta}\po -2\alpha  + \eta(2\eta+d+4)\alpha  + e^{-t}\pf \\
& \geq & -\varepsilon^2\| H^{-2\eta}\na^2 U\|_{\infty} - 3\varepsilon\po \eta+d\pf^2
\end{eqnarray*}
\begin{eqnarray*}
| b+c-a\na^2U-(L_s+\partial_t)b|
& \leq & |\varepsilon^2 \alpha^2  H^{-2\eta} + 2\varepsilon \alpha  H^{-\eta} - 2 e^{-t} \varepsilon^2 \alpha  H^{-2\eta}| \\
& & + |\varepsilon^3 \alpha^3  H^{-3\eta}\na^2 U  | + 2\eta (4\eta+d+4) \varepsilon^2 \alpha^2  H^{-2\eta}\\
&\leq & \varepsilon\alpha H^{-\eta} \po \varepsilon^2\| H^{-2\eta}\na^2 U\|_{\infty} + 2 + 8\varepsilon(\eta+d)^2\pf
\end{eqnarray*}
which implies for $\varepsilon = \frac14  \times \frac19  (\eta+d)^{-2}$ that
\begin{eqnarray*}
(\na f)^T N \na f & \geq &  \frac14\varepsilon^2\alpha^2 H^{-2\eta} |\na_x f|^2
  - A|\na_y f|^2
\end{eqnarray*}
with
\begin{eqnarray*}
A &=& \frac12 \po \varepsilon^2\| H^{-2\eta}\na^2 U\|_{\infty} + 2 + \frac{2}{9}\pf^2 + \varepsilon^2\| H^{-2\eta}\na^2 U\|_{\infty} + \frac{1}{12} \\
& \leq &  \po \| H^{-2\eta}\na^2 U\|_{\infty} + 2\pf^2\ :=\ \lambda.
\end{eqnarray*} 
%Recall we defined $\lambda = 35 \po \| H^{-2\eta}\na^2 U\|_{\infty} + (\eta+d)^2\pf^2$.
Writing
\begin{eqnarray*}
G(t) &=& \frac{1}{2\lambda} F(t) + \int \Psi(P_t f)\dd \mu,
\end{eqnarray*}
we have obtained
\begin{eqnarray*}
G'(t) &\leq & -  \int \Psi''(P_t f) \po \frac{\alpha^2 \varepsilon^2}{4\lambda} H^{-2\eta}|\na_x P_t f|^2+ \po 2 - \frac A\lambda \pf|\na_y P_t f|^2\pf   \dd \mu \\
&\leq & - \frac{\alpha^2 \varepsilon^2}{4\lambda} \int \Psi''(P_t f) \po  H^{-2\eta}|\na_x P_t f|^2+  |\na_y P_t f|^2\pf   \dd \mu.
\end{eqnarray*}
On the one hand,
\begin{eqnarray*}
F(t) & \leq & 3\varepsilon \alpha  \int \Psi''(P_t f) \po H^{-2\eta}|\na_x P_t f|^2+|\na_y P_t f|^2\pf   \dd \mu,
\end{eqnarray*}
and on the other hand, using the inequality \eqref{assump2},
\begin{eqnarray*}
\int \Psi(P_t f)\dd \mu  & \leq & \rho  \int \Psi''(P_t f) \po H^{-2\eta}|\na_x P_t f|^2+|\na_y P_t f|^2\pf   \dd \mu,
\end{eqnarray*}
which implies
\begin{eqnarray*}
G'(t) & \leq &  - \frac{\alpha^2 \varepsilon^2}{1+4\lambda\rho} G(t).
\end{eqnarray*}
Hence,
\[\text{Ent}_\mu(P_t f) \ \leq \ G(t) \ \leq \ G(0) \exp\po- \frac{ \varepsilon^2}{1+4\lambda\rho}\int_0^t \alpha^2(s)\dd s\pf,\]
and $G(0) = \text{Ent}_\mu(f)$. The proof is complete.
\end{proof}
\medskip

%%%%%%%%%%%%%%%%%%%
%%%%%%%%%%%%%%%%%%%

\section{Weighted Functional Inequalities with $\eta\geq 0$.}
We turn to the study of the functional inequality \eqref{assump2}. For simplicity we shall only consider the cases $\Psi(u)=(u-1)^2$ (Variance) and $\Psi(u)=u\ln u -u +1$ (Entropy). \\ \noindent Recall the definition of $L_\eta$, $$L_{\eta}:=H^{-2\eta} \Delta_x +  \Delta_y - H^{-2\eta}\left(2\eta\frac{\nabla_x H}{H}+\nabla_x H\right).\nabla_x - \nabla_y H.\nabla_y \, ,$$ which satisfies
\begin{equation}\label{eqIPP'}
-\int \, f \, L_\eta f \, \dd \mu \, =  \, \int \, (H^{-2\eta} \, |\nabla_x f|^2+ |\nabla_y f|^2) \, \dd \mu \, : = \mathcal E_\eta(f).
\end{equation}
Let us state our first main results
\begin{thm}\label{thmlyap-var}
The weighted Poincar\'e inequality $$\text{Var}_\mu(g) \leq \rho \, \int \, \left(H^{-2\eta}|\nabla_x g|^2+|\nabla_y g|^2\right) \, \dd\mu$$ is satisfied if and only if there exists a Lyapunov function, i.e. a smooth function $W$ such that $W(x,y)\geq w >0$ for all $(x,y)$, a constant $\lambda>0$ and a bounded open set $A$ such that $$L_\eta W \leq - \, \lambda \, W \, + \, \mathbf 1_{\bar A} \, .$$
\end{thm}
We provide then the equivalent result for the logarithmic Sobolev inequality.
\begin{thm}\label{thmlyap-ent}
Assume that $H$ goes to infinity at infinity and that there exists $a>0$ such that $e^{aH}\in \mathbb L^1(\mu)$. 
\begin{enumerate}
\item If $\mu$ satisfies the weighted log-Sobolev inequality \eqref{eqlspoids}, then, there exists a Lyapunov function, i.e. a smooth function $W$ such that $W(x,y)\geq w >0$ for all $(x,y)$, two positive constants $\lambda$ and $b$ such that 
\begin{equation}\label{eqlyapls}
L_\eta W \leq - \, \lambda \, H \, W \, + \, b \, .
\end{equation}
\item Conversely, assume that there exists a Lyapunov function satisfying \eqref{eqlyapls} and that $|\nabla H|(x,y) \geq  c > 0$ for $|(x,y)|$ large enough. Define \[
\theta(r)=\sup\limits_{z\in \partial A_r} \max\limits_{i,j=1,...,2d} |\frac{\partial^2 H}{\partial z_i\partial z_j} |
\] 
and assume that $\theta(r)\leq ce^{C_0r}$ with some positive constants $C_0$ and $c$  for $r$ sufficiently large. Then $\mu$ satisfies the weighted log-Sobolev inequality \eqref{eqlspoids}.
\end{enumerate}
\end{thm}
These theorems are the analogues, in the weighted situation we are looking at, of (part of) Theorem 1.1 and Theorem 1.2 in \cite{CG16}. Their proofs are very similar concerning the part 1) of the previous theorem and we shall only give some details in the entropic case. Let us begin by a  simple and crucial Lemma, at the basis of the use of Lyapunov type condition. Note that it can also be proved via large deviations argument.

\begin{lem}\label{lem52} For every continuous  function $W\ge 1$ in the domain of $L_\eta$ such
that $-L_\eta W/W$ is $\mu$-a.e. lower bounded, for all $g $ in the domain of $L_\eta$,
\begin{equation}\label{lem52a}
\int -\frac{L_\eta W}{W} g^2 \, \dd\mu \le \int \left( H^{-2\eta} |\nabla_x g|^2 +  |\nabla_y g|^2 \right) \, \, \dd\mu.\end{equation}
\end{lem}
\begin{proof}This follows from integration by parts and Cauchy-Schwartz inequality. Indeed,
\begin{eqnarray*}
\int -\frac{L_\eta W}{W} g^2 \,\dd\mu
&=& \int H^{-2\eta}\langle \nabla_x W,\nabla_x \frac{g^2}{W}\rangle + \langle \nabla_y W,\nabla_y \frac{g^2}{W}\rangle \dd\mu\\
&=&\int H^{-2\eta}\left(-\frac{g^2}{W^2}|\nabla_x W|^2 + 2\frac{g}{W}\langle \nabla_x W,\nabla_xg\rangle\right)\\&&\qquad\qquad\qquad+ \left(-\frac{g^2}{W^2}|\nabla_y W|^2 + 2\frac{g}{W}\langle \nabla_y W,\nabla_yg\rangle\right) \dd\mu \\
&\leq& \int \left( H^{-2\eta} |\nabla_x g|^2 +  |\nabla_y g|^2 \right) \, \, \dd\mu
\end{eqnarray*}
\end{proof} 

Let us now prove Theorem \ref{thmlyap-ent}.

\begin{proof}
For a given nice function $\phi$, introduce the operator $G_\eta$ via $G_\eta h=- \, L_\eta h +\phi h$. For any $h$ in the domain of $L_\eta$, $\int \, h \, G_\eta h \dd \mu = \mathcal E_\eta(h) + \int h^2 \, \phi \, \dd\mu$. Choosing $\phi = -c + \mathbf 1_A$ for some set $A$ to be defined, in the variance case and $\phi = \rho(b - H)$ in the entropic case, one deduces that $G_\eta$ is continuous for the norms whose square are respectively $\mathcal E_\eta(h) + \int_A h^2 \, \dd\mu$ and $\mathcal E_\eta(h) + \int h^2  \, \dd\mu$. If a weighted Poincar\'e inequality (resp. weighted log-Sobolev inequality) is satisfied, following the proof of Theorem 2.1 (resp. Proposition 3.1) in \cite{CG16}, we get that the form $\int h \, G_\eta h \, \dd\mu$ is also coercive so that applying Lax-Milgram theorem we get a solution to $G_\eta h= 1$, which furnishes the desired Lyapunov function (see \cite{CG16} for the details).
\medskip

For the converse, we revisit the proof of \cite{CG16} Proposition 3.5 in order to adapt it to our case. As usual, we will rather prove the (weighted) log-Sobolev inequality in its equivalent (weighted) Super Poincar\'e inequality form, i.e. there exist $c,\beta>0$ such that for all smooth $f$ and $s>0$, 
$$\int f^2\dd\mu\le s\int (H^{-2\eta}|\nabla_x f|^2+|\nabla_y f|^2)\dd\mu+c\,e^{\beta/s}\left(\int|f|\dd\mu\right)^2.$$ Indeed, the latter implies a defective (weighted) log-Sobolev inequality and a weighted Poincar\'e inequality (choosing $s$ such that $c e^{\beta/s}=1$) and we obtain a tight (weighted) log-Sobolev inequality by using Rothaus lemma (see \cite{BGL14} p.239), which states that
\begin{equation}\label{eqrot}
\mbox{Ent}_\mu(f^2) \leq \mbox{Ent}_\mu(\tilde f^2) + 2 \mbox{Var}_\mu(f) \, ,
\end{equation}
where $\tilde f= f - \int f \, \dd\mu$. For all this we refer to \cite{CGWW09,CGW11,Wbook}.
\medskip

Recall $A_r=\{H \leq r\}$. For $r_0$ large enough and some $\lambda'<\lambda$ we have
$$L_{\eta}W \leq - \lambda' \, H \, W \, + \, b \, \mathbf 1_{A_{r_0}} \, ,$$
so that we may assume that
$$\frac{L_{\eta}W}{W}(x,y) \leq - \, \lambda \, H(x,y) \, + \, \frac{b}{w} \, \mathbf 1_{A_{r_0}} \, ,$$
We have for %$s \leq s_0$ \textcolor{red}{(pas de $s$ ci-dessous ?)} and
 $r>r_0$,
\begin{eqnarray*}
\int \, f^2 \, \dd\mu
&\leq& \int_{A_r} f^2 \, \dd\mu \, + \, \int_{A^c_r} \frac{\lambda H}{\lambda r}f^2 \, \dd\mu \\
&\leq& \int_{A_r} f^2 \, \dd\mu \, + \, \int \frac{\lambda H}{\lambda r}f^2 \, \dd\mu \\
&\leq& \int_{A_r} f^2 \, \dd\mu \, + \, \frac{1}{\lambda \, r} \, \int \left(\frac{-L_{\eta} W}{W}\, +\, \frac{\, b \, \mathbf 1_{A_{r_0}} \,}{w}\right) \,f^2 \, \dd\mu\\
&\leq& \left(1+\frac{b}{\lambda r w}\right) \, \int_{A_r} f^2 \, \dd\mu \, + \,  \frac{1}{\lambda \, r} \, \int \left( H^{-2\eta} |\nabla_x f|^2 +  |\nabla_y f|^2 \right) \, \, \dd\mu \\
\end{eqnarray*}
It remains to control the integral in $A_r$. It is in fact a simple consequence of Nash inequalities for the Lebesgue measure rewritten in its Super Poincar\'e form (c.f. \cite[Prop 3.8]{CGWW09}): there exists $c_d$ such that for all $r$ large enough, all smooth $f$ and $s>0$
\begin{eqnarray*}
  \int_{A_r} f^2 \dd x \dd y &\le& s \, \int_{A_r} |\nabla f|^2 \dd x \dd y+c_d \theta^d(r)(1+s^{-2d})\left(\int|f| \dd x \dd y\right)^2 \\
  &\leq & s\,\int_{A_r} |\nabla f|^2 \dd x \dd y+c_d c e^{2dC_0r}(1+s^{-2d})\left(\int|f| \dd x\dd y\right)^2 \, .
\end{eqnarray*}
Recall that $H\geq 1$. We thus have
\begin{eqnarray*}
\int_{A_r} \, f^2 \, \dd\mu &\leq& \frac{1}{eZ} \, \int_{A_r} f^2 \dd x \dd y \\ &\leq& \frac{r^{2\eta} \, e^r}{e} \, s \, \int \left( H^{-2\eta} |\nabla_x f|^2 +  |\nabla_y f|^2 \right)  \, \dd\mu \, + \, Z c_dc e^{2dC_0r}(1+s^{-2d})e^{2r}\left(\int_{A_r} |f|d\mu\right)^2 \, .
\end{eqnarray*}
Letting $u=se^{r-1} \, r^{2\eta}$ and $C'=Zcc_d$, and considering integral on the whole space in the right hand side, we have thus obtained (for $r$ large enough)
\begin{eqnarray*}
\int_{A_r} \, f^2 \, \dd\mu &\leq&  u \, \int \left( H^{-2\eta} |\nabla_x f|^2 +  |\nabla_y f|^2 \right)  \, \dd\mu \, + \,C' \, r^{4d\eta} \, (1+u^{-2d})e^{2(1+dC_0+d)r}\left(\int|f|d\mu\right)^2 \, .
\end{eqnarray*}
Denoting $c=1+\frac{b}{\lambda r_0 w}$, and $\beta_d=2+d+2dC_0$, we thus have, for all $u>0$ and $r$ large enough 
\begin{equation}\label{eqsuperP}
\int f^2 \dd\mu \, \leq \, \left(u \, c \, + \, \frac{1}{\lambda \, r}\right) \, \int \left( H^{-2\eta}|\nabla_x f|^2 + |\nabla_y f|^2 \right) \, \, \dd\mu \, + \, C' \, (1+u^{-2d}) \, r^{2d\eta} \, c \, e^{\beta_d r} \,  \left(\int \,  |f| \, d\mu\right)^2.
\end{equation}
Choosing $r\lambda =(uc)^{-1}$ and $s=2uc$, we have thus proved the existence of some $\beta'_d$ such that
$$\int f^2 \dd\mu \, \leq \, s \, \int \left( H^{-2\eta}|\nabla_x f|^2 + |\nabla_y f|^2 \right) \, \, \dd\mu \, + \, C'' \, e^{\beta'_d/s} \,  \left(\int \,  |f| \, d\mu\right)^2 \, ,$$ and the proof is complete.
\end{proof}

\begin{rmq}
For a general weighted logarithmic Sobolev inequality with the weighted energy $$\int \left( w_1|\nabla_x f|^2 + w_2|\nabla_y f|^2 \right)d\mu,$$ we can introduce the symmetric generator
$$L_{w_1,w_2}:=w_1 \Delta_x + w_2 \Delta_y - w_1\left(-\frac{\nabla_x w_1}{w_1}+\nabla_x H\right).\nabla_x - w_2\left(-\frac{\nabla_y w_2}{w_2}+\nabla_y H\right).\nabla_y.$$
If a Lyapunov function (as in Theorem \ref{Thm-LyapunovCondition} but for $L_{w_1,w_2}$) exists, then following the same line, we can obtain (with the retired additional assumptions on the weights) a weighted logarithmic Sobolev inequality. \hfill $\diamondsuit$
\end{rmq}

We now proceed to the
\begin{proof}[Proof of Corollary \ref{Cor-Lyapunov}]
Consider a smooth function $W(x,y)=e^{\alpha U(x)+\frac{\beta}{2}|y|^2}$ with two constants $\alpha, \beta\in (0,1)$ to be determined. Then for $|(x,y)|\geq R$,
\begin{eqnarray*}
\frac{L_\eta W}{W}
&=& \alpha H^{-2\eta}\left[ \Delta_x U +\left( \alpha - \frac{2\eta}{H}-1 \right)|\nabla_x U|^2\right] + \beta(d-(1-\beta)|y|^2)\\
&\le& \beta d- \alpha \left( 1 - \alpha - \kappa \right)|\nabla_x U|^2 H^{-2\eta}-\beta(1-\beta)|y|^2
\end{eqnarray*}
where we used the first condition in the assumption of the corollary.

To bound the last term by some $C -\lambda H$, we consider $\alpha \in (0,1-\kappa), \beta\in (0,1)$, and divide it into two cases. If $\frac{|y|^2}{2}\geq \frac{H}{2}$, then
\[
-\alpha\left( 1 - \alpha - \kappa \right)|\nabla_x U|^2 H^{-2\eta}-\beta(1-\beta)|y|^2
\le -\beta(1-\beta) H
\]
Otherwise,we have $U\geq \frac{H}{2}$. Combined with the second condition, it follows
\[
-\frac{|\na_x U|^2}{H^{2\eta}}\leq -\frac{c U^{2\eta+1}}{2^{2\eta}U^{2\eta}}\leq -\frac{c}{2^{2\eta+1}}H
\]
which completes the proof of the Lyapunov condition. Since the second condition implies that $U$ goes to infinity at infinity and $|\nabla_x U|\geq u\geq 0$, we get a weighted logarithmic Sobolev inequality for $\mu$ by the previous theorem.
\end{proof}

The next example, which is the simple polynomial case will show the adequacy of our conditions on weighted log-Sobolev inequality with the Assumption \ref{HypB}.

\begin{ex}
Let us consider the example where $U(x)=|x|^l$ with $l>2$ for $|x|$ large enough, that is, $H(x,y)=|x|^l+\frac{|y|^2}{2}$. Then $\Delta_x U=(dl+l^2-2l))|x|^{l-2}$ and $|\nabla_x U|^2=l^2 |x|^{2l-2}$. The first condition is satisfied since $l> 2$, while for the second condition we need
$$\eta \le \frac{1}{2}-\frac{1}{l}$$
Note that $||U^{-2\eta}\nabla^2 U||_{\infty}\sim |x|^{l-2-2l\eta}$, to ensure that $U^{-2\eta}\nabla^2 U$ is bounded, we have to choose $\eta=\frac{1}{2}-\frac{1}{l}$. With the case $l=2$ we recover Villani's result.
\end{ex}
Let us give another example which will show that our limit growth for the potential $U$ is below the exponential growth
\begin{ex}
Choose now $U(x)=e^{a|x|^b}$ for $a,b>0$ for $|x|$ large enough. Then $\Delta_xU\sim a^2b^2|x|^{2(b-1)}e^{a|x|^b}$ and $|\nabla_x U|^2\sim a^2b^2 e^{2a|x|^b}$. The first condition is thus satisfied , while the second one imposes once again that $2\eta+1\le 2$. Now, Assumption 1 imposes that $2\eta>1$ if $b\ge 1$ leading to an impossible adequacy of the two sets of conditions and to $2\eta\ge 1$ if $b<1$ and thus the choice of $\eta=1/2$ is admissible.
\end{ex}

Let us end this section by a remark
\begin{rmq}
For the multipliers method in the variance case, Villani does not use $H^{-2\eta}$ in the energy to get his inequality, as will be seen in the next section but prove a rather stronger inequality with weight in the derivative in x in the energy $U(x)^{-2\eta}(1+|y|^2)^{-2\eta}$. The fact that he deals with the variance helps him enough to prove such a weighted Poincar\'e inequality. We may also consider a weighted logarithmic Sobolev inequality with such a weight. However, via the Lyapunov condition approach, the condition on $\eta$ is then too strong to match with Assumption 1. It is thus crucial to have a weighted inequality with weight $H^{-2\eta}$ for Theorem~1.
\end{rmq}

The next section presents an alternative approach trying to provide an answer to the problem alluded in the previous remark. Is it possible to provide a "tensorization-like" approach to provide a weighted logarithmic Sobolev inequality as in Villani's paper, thus giving an alternative to Lyapunov conditions?

\medskip

\section{Some further remarks on weighted inequalities.}\label{sec comments}

In this final section we shall try to understand whether it is possible to impose conditions on $U$ solely in order to get weighted inequalities. We shall use several times the following elementary inequalities, true for all $\eta \geq 0$, all $x$ and $y$ (recall that $U\geq 1$)
\begin{equation}\label{eqneqH}
U^{-\eta}(x) \, \left(1+\frac 12 \, |y|^2\right)^{-\eta} \, \leq \, H^{-\eta}(x,y) \, \leq \min \left(U^{-\eta}(x) \, , \, \left(1+\frac 12 \, |y|^2\right)^{-\eta}\right) \, .
\end{equation}
We shall use in the sequel the notations $U^{- 2\eta}(x)=\phi_1(x)$ and $\left(1+\frac 12 \, |y|^2\right)^{-2 \eta}=\phi_2(y)$.
\medskip

\subsection{The case of weighted Poincar\'{e} inequalities. \\ \\}

Assume that $\mu$ satisfies a weighted Poincar\'e inequality. If we choose an $f$ that only depends on $x$ and use that $H^{-2\eta}(x,y)\leq U^{-2\eta}(x)$ for all $y$, we immediately see that the first marginal of $\mu$, i.e. $\dd\mu_1(x) :=\frac{1}{Z_1}e^{-U(x)}\dd x $ also satisfies the weighted Poincar\'{e} inequality
\begin{equation}\label{Eq-weighted PI on x}
 \text{Var}_{\mu_1}(f) \leq C \int U^{-2\eta}|\nabla f|^2\dd\mu_1 \, .
\end{equation}
Conversely we have, 
\begin{thm}\label{thmwp} Write $\mu(dx,dy)=\mu_1(dx)\otimes \mu_2(dy)$. If $\mu_1(dx) =\frac{1}{Z_1} \, e^{-U(x)}\dd x$ satisfies the weighted Poincar\'{e} inequality \eqref{Eq-weighted PI on x} with constant $C_1$,
then $\mu$ satisfies the following weighted Poincar\'{e} inequality
\[
\text{Var}_{\mu}(h) \leq C' \int (H^{-2\eta}|\nabla_x h|^2 + |\nabla_y h|^2)\dd\mu
\]
with $$C'\leq \max \left(\left(2+\frac{4}{M_2}\right), \frac{4C_1}{M_2}\right) \quad \textrm{ where } M_2= \int \, \left(1+ \frac 12 |y|^2\right)^{-2\eta} \, \mu_2(dy)  \, .$$
\end{thm}

\begin{proof}
A proof is given in  Villani \cite{Villani} Theorem A.3. It uses extensively the spectral theory of the sum of operators. We shall give a more pedestrian (similar) proof.
\medskip

\noindent The first point is that, since we assumed that $U\geq 1$,
\begin{equation}\label{eqminH}
H^{-2\eta}(x,y) \geq \phi_1(x) \, \phi_2(y) := \, U^{-2\eta}(x) \; \left(1 + \frac 12 \, |y|^2\right)^{-2\eta} \, .
\end{equation}
Thus, if we decompose $\mu(dx,dy)=\mu_1(dx)\otimes \mu_2(dy)$ we have
\begin{eqnarray*}
\int \, H^{-2\eta} \, |\nabla_xh|^2 \, \mu(dx,dy) &\geq& \int \, \phi_1(x) \, \phi_2(y) \, |\nabla_xh|^2 \, \mu_1(dx)\otimes \mu_2(dy) \\ &\geq& \frac{1}{C_1} \, \int \, \phi_2(y) \, \left(h(x,y)-\int h(u,y)\mu_1(du)\right)^2 \, \mu(dx,dy) \, .
\end{eqnarray*}
Now write, 
\begin{eqnarray*}
h(x,y)-\int h(u,y)\mu_1(du) &=& \left(h(x,y)-\int h(u,y)\mu_1(du) -\int h(x,v)\mu_2(dv)+ \int\int h \dd\mu_1 \dd\mu_2\right) + \\ &+&\left(\int h(x,v) \mu_2(dv) -\int\int h \dd\mu_1 \dd\mu_2\right)\\ &=& g_1(x,y) + g_2(x)
\end{eqnarray*}
\noindent and use $$(a+b)^2 \geq \frac 12 \, b^2 \, - \, a^2 \, .$$ This yields, since $\phi_2(y) \leq 1$,
\begin{equation*}
\int \, H^{-2\eta} \, |\nabla_xh|^2 \, \mu(dx,dy) \geq \frac{1}{2C_1} \, \left(\int \phi_2 \, d\mu_2\right) \left(\int g_2^2(x) \mu_1(dx)\right) - \, \frac{1}{C_1}  \, \int\int g_1^2(x,y) \, \mu_1(dx)\mu_2(dy) \, .
\end{equation*}
Notice that for all $y$, $$\int g_1^2(x,y) \, \mu_1(dx)= \text{Var}_{\mu_1}\left(h(.,y)- \int h(.,v) \, \mu_2(dv)\right)$$ so that $$\int g_1^2(x,y) \, \mu_1(dx) \leq \int \, \left(h(x,y)- \int h(x,v) \, \mu_2(dv)\right)^2 \, \mu_1(dx) \, .$$
We can thus integrate this inequality w.r.t. $\mu_2$, use Fubini's theorem, then for each fixed $x$ use the usual Poincar\'e inequality for the standard gaussian measure $\mu_2$ and finally integrate with respect to $\mu_1$. This yields 
\begin{eqnarray*}
\int\int g_1^2(x,y) \, \mu_1(dx)\mu_2(dy) &\leq& \int\int \, \left(h(x,y)- \int h(x,v) \, \mu_2(dv)\right)^2 \, \mu(dx,dy) \\  &\leq& \int\int |\nabla_y h|^2(x,y) \, \mu(dx,dy) \, .
\end{eqnarray*} 
Gathering all this we have obtained
\begin{equation}\label{eqpwinter}
\int g_2^2(x) \mu_1(dx) \leq \frac{2 \, C_1}{M_2} \, \int H^{-2\eta}|\nabla_x h|^2 \, \dd\mu + \frac{2}{M_2} \, \int |\nabla_y h|^2 \, \dd\mu \, .
\end{equation}
Finally
\begin{eqnarray*}
\text{Var}_{\mu}(h) &=& \int \, \left(h(x,y)-\int h(x,v)\mu_2(dv)+\int h(x,v)\mu_2(dv)-\int \int h\dd \mu\right)^2 \, \mu(dx,dy) \\ &\leq& 2 \,\int\int \,  \left(h(x,y)- \int h(x,v) \, \mu_2(dv)\right)^2 \, \mu(dx,dy) + \, 2 \, \int g_2^2(x) \mu_1(dx) \\ &\leq& 2 \, \int\int |\nabla_y h|^2(x,y) \, \mu(dx,dy) + \, 2 \,  \int g_2^2(x) \mu_1(dx) \, ,
\end{eqnarray*}
and the result follows from \eqref{eqpwinter}.
\end{proof}

\noindent As a conclusion the weighted Poincar\'e inequality on $\mathbf R^{2d}$ reduces to a weighted Poincar\'{e} inequality on $\R^d$ (up to some constant). One should think that the previous result is a kind of weighted tensorization property. This is not the case due to the fact that the weight in front of $\nabla_x$ depends on both variables $x$ and $y$.\\ \noindent There are many ways to obtain such an inequality. Of course since it is stronger than the usual Poincar\'e inequality, our result is weaker than the one of Villani (but with a simpler proof and explicit bounds for the constants), and we will only describe a typical situation where this equality can be obtained. \\ As we have seen in the previous section, this weighted Poincar\'e inequality is equivalent to the existence of some Lyapunov function for $L_{1,\eta}$ which is built similarly to $L_\eta$ replacing $H$ by $U$. We can also obtain a slightly different condition. Introduce the probability measure $\mu_1^\phi(dy) = \frac{\phi_1(y)}{M_1} \, \mu_1(dy)$ and the $\mu_1^\phi$ symmetric operator $$G_1^\phi = \Delta_x - \left(1+\frac{2\eta}{U}\right) \nabla U. \nabla \, .$$ Assume that we can find a Lyapunov function $W\geq 1$ such that $$\frac{G_1^\phi W(x)}{W(x)} \leq - a \, U^{2\eta}(x)$$ for $|x|$ larger than some $R>0$. If $h$ is compactly supported in $|x|>R$, we may write $$\int h^2 \, d\mu_1 \leq - \frac{M_1}{a} \, \int \frac{G_1^\phi W}{W} \, h^2 \, \dd \mu_1^\phi \leq \frac{M_1}{a} \int |\nabla h|^2 \, \dd \mu_1^\phi = \frac{M_1}{a} \, \int |\nabla h|^2 \, U^{-2\eta} \, \dd \mu_1$$ according to the computations in \cite{BBCG} p.64. Following the method introduced in \cite{BBCG} we then obtain that $\mu_1$ satisfies the desired weighted Poincar\'e inequality. According to \cite{CG16} Theorem 4.4, the existence of such a Lyapunov function is linked to the fact that $\mu_1$ satisfies some $F$-Sobolev inequality, with $F=\ln_+^{2\eta}$. This is for instance the case when $U(x)=1+|x|^\alpha$ and $\eta=1-\alpha^{-1}$.
\bigskip

\subsection{The case of weighted log-Sobolev  inequalities. \\ \\}
\noindent We look now at the similar weighted logarithmic Sobolev inequality, namely,
\begin{equation*}
\mbox{Ent}_\mu(f^2)\leq \rho \int  (H^{-2\eta}|\na_x f|^2 + |\na_y f|^2)\dd \mu.
\end{equation*}
\noindent As in the $L^2$ setting, it implies a weighted log Sobolev inequality for $\mu_1$ on $\R^d$ i.e.
\begin{equation}\label{eqls1}
\mbox{Ent}_{\mu_1}(f^2)\le C \, \int U^{-2\eta}|\nabla_x f|^2\dd\mu_1 \, .
\end{equation} 
Since the standard gaussian measure $\mu_2$ satisfies a log-Sobolev inequality too (with optimal constant $2$), one should expect to obtain the analogue of theorem \ref{thmwp}. This is not so easy (actually we did not succeed in proving such a result) and certainly explains the limitation of Villani's approach, since this property reduces to the well known tensorization property of the logarithmic Sobolev inequality only in the case $\eta=0$. The best we are able to do is to prove that, in this situation 
\medskip

\begin{thm}\label{thmtensorls}
Write $\mu(dx,dy)=\mu_1(dx)\otimes \mu_2(dy)$. If $\mu_1(dx) =\frac{1}{Z_1} \, e^{-U(x)}\dd x$ satisfies the weighted log-Sobolev inequality \eqref{eqls1}, then $\mu$ satisfies \eqref{assump2} with an admissible function $u\mapsto \Psi(u)$ behaving like $u \, \ln^{\frac 12}(u)$ at infinity.
\end{thm}

Combined with the results of Section 2 which deals with a decay for more general functionals than the variance or entropy, we are thus able to prove under such conditions an exponential decay for $\Psi$ behaving like $u \, \ln^{\frac 12}(u)$ at infinity.

\begin{proof}
The first step of the proof is the following 
\begin{lem}\label{lemphils}
Define the probability measure $\mu_2^\phi(dy) = \frac{\phi_2(y)}{M_2} \, \mu_2(dy)$. Then $\mu_2^\phi$ satisfies a log-Sobolev inequality.
\end{lem}
\noindent An immediate consequence is the following inequality for $\mu^\phi(dx,dy)=\mu_1(dx) \otimes \mu_2^\phi(dy)$,
\begin{equation}\label{eqtens}
\text{Ent}_{\mu^\phi}(h^2) \leq C \, \int (\phi_1 \, |\nabla_x h|^2 + |\nabla_y h|^2)\dd\mu^\phi \, , 
\end{equation}
which follows from the tensorization property of the log-Sobolev inequality.
\begin{proof}[Proof of Lemma \ref{lemphils}]
Write $$\mu_2^\phi(dy) = Z^\phi \, e^{- \, \left(\frac{|y|^2}{2} + 2\eta \, \ln(1+|y|^2/2)\right)} \dd y= Z^\phi \, e^{-V_2(y)} \dd y \, .$$ A simple calculation shows that $$Hess V_2(y)= \left(1 + \frac{2\eta}{1+|y|^2/2}\right) \, Id \; - \; \frac{2\eta}{(1+|y|^2/2)^2} \, M(y)$$ where $M_{i,j}(y)=y_i y_j$. Hence, $$Hess  V_2(y) \, \geq \, \left(1 + \frac{2\eta}{1+|y|^2/2} - \frac{2 \eta d \, |y|^2}{(1+|y|^2/2)^2}\right) \, Id$$ in the sense of quadratic forms. Hence for $|y|$ large enough (of order $c \sqrt n$), the potential $V_2(y)$ is uniformly convex, uniformly in $y$. This proves (combining Bakry-Emery criterion and Holley-Stroock perturbation argument) the Lemma.
\end{proof}

As we recalled, the weighted log-Sobolev inequality is equivalent to a (weighted) super Poincar\'e inequality, for all smooth $h$ and all $s>0$,
\begin{equation}\label{eqsuperdef1}
\int h^2 \dd\mu^\phi \, \leq \, s\int (\phi_1 \, |\nabla_x h|^2+|\nabla_y h|^2) \dd\mu^\phi \, + \, c\, e^{\beta/s}\left(\int|h| \, \dd\mu^\phi\right)^2 \, .
\end{equation}
Since $\phi_2 \leq 1$, it follows
\begin{equation}\label{eqsuperdef2}
\int h^2 \dd\mu^\phi \, \leq \,  \frac s{M_2} \int (H^{-2\eta} \, |\nabla_x h|^2+|\nabla_y h|^2) \dd\mu \, + \,  \frac c{M_2} \, e^{\beta/s}\left(\int|h| \, \dd\mu\right)^2 \, .
\end{equation}
For $R>1$, introduce the 1-Lipschitz function $$\varphi(r)=(r-R) \, \mathbf 1_{R<r<R+1} + \mathbf 1_{R+1\leq r} \, .$$ One can write 
\begin{eqnarray*}
\int h^2 \dd\mu &\leq& \int_{|y|\leq R+1} \, h^2 \, \dd\mu + \int \, h^2 \, \varphi^2(|y|) \, \dd\mu \\ &\leq& 
 \frac{M_2}{\phi_2(R+1)} \, \int_{|y|\leq R+1} h^2 \dd\mu^\phi \, + \, \int \, h^2 \, \varphi^2(|y|) \, \dd\mu \\ &\leq&  \frac{M_2}{\phi_2(R+1)} \, \int h^2 \dd\mu^\phi \, + \, \int \, h^2 \, \varphi^2(|y|) \, \dd\mu \, .
\end{eqnarray*} 
The first term in the sum will be controlled thanks to \eqref{eqsuperdef2}. In order to control the second term, we introduce,once again, some Lyapunov function. \\ 
\noindent Denote by $G$ the Ornstein-Uhlenbeck operator $G=\Delta_y - y.\nabla_y$ and consider $W(y)=e^{|y|^2/4}$. A simple calculation shows that $$\frac{GW}{W} \leq \frac 14 \, (2d-|y|^2)$$ for $|y|>\sqrt{2d}$. Hence if $R>\sqrt{2d}$, we get for $|y|>R$, $$1 \leq 4 \, \left(\frac{-GW}{W}\right) \, \frac{1}{|y|^2-2d} \leq 4 \, \left(\frac{-GW}{W}\right) \, \frac{1}{R^2-2d}$$ and finally
\begin{equation}\label{eqlyap}
\int \, h^2 \, \varphi^2(|y|) \, \dd\mu \leq \frac{4}{R^2-2d} \, \int \, \left(\frac{-GW}{W}\right) \, h^2 \, \varphi^2(|y|) \, \dd\mu \, .
\end{equation}
\noindent Integrating by parts, and after some easy manipulations (see \cite{BBCG} for the details), we will thus obtain for well chosen constants $C,C'$ all $s>0$ and large enough $R$, 
\begin{equation}\label{eqlsbeta}
\int h^2 \dd\mu \leq C \, (sR^2+R^{-2}) \int (\phi_1 \, |\nabla_x h|^2+|\nabla_y h|^2) \dd\mu^\phi \, + \, C' \, R^2 \, e^{\beta/s}\left(\int|h| \, \dd\mu\right)^2 \, .
\end{equation}
Choosing $u=s^{\frac 12}$ and $R=s^{-1/4}$, we obtain a super Poincar\'e inequality 
\begin{equation}\label{eqlsbeta2}
\int h^2 \dd\mu \leq C \, u \int (\phi_1 \, |\nabla_x h|^2+|\nabla_y h|^2) \dd\mu^\phi \, + \, C'  \, e^{\beta'/u^2}\left(\int|h| \, \dd\mu\right)^2 \, .
\end{equation}
which furnishes a $F=\ln_+^{\frac 12}$-Sobolev inequality, i.e. if $\int h^2 \, \dd \mu=1$,
\begin{equation*}
\int h^2 \, \ln_+^{\frac 12} h^2 \, \dd\mu \leq C \, \int (\phi_1 \, |\nabla_x h|^2+|\nabla_y h|^2) \dd\mu^\phi  \, .
\end{equation*}
Notice that, since $\phi_2\leq 1$, the previous inequality is stronger than 
\begin{equation}\label{eqlsbeta3}
\int h^2 \, \ln_+^{\frac 12} h^2 \, \dd\mu \leq C \, \int (H^{-2\eta} \, |\nabla_x h|^2+|\nabla_y h|^2) \dd\mu \, .
\end{equation}
It remains to link \eqref{eqlsbeta3} to \eqref{assump2}. Actually, as explained in \cite{BCR06} section 7, one can replace $\ln_+$ by smooth functions $F$ with a similar behaviour at infinity (and satisfying $F(1)=0$. So we choose $\psi(u)= \frac{\ln^{\frac 12}(e+u)}{u}$ and $\Psi''=\psi$ with $\Psi(1)=0$. $\Psi(u)$ behaves like $F(u)= u \ln^{\frac 12}(e+u)$ at infinity. Applying \eqref{eqlsbeta3} with $\Psi$ instead of $u \ln_+^{\frac 12}(u)$ (modifying the constant) and $h^2=f$ we have (the value of $C$ varies from one line to the other)
\begin{eqnarray*}
\int \, \Psi(f) \, \dd\mu \, &\leq& \, C \, \int \frac 1f (H^{-2\eta} \, |\nabla_x f|^2+|\nabla_y f|^2) \dd\mu \, \\ &\leq&  C \, \int \frac{\ln^{\frac 12}(e+f)}{f} (H^{-2\eta} \, |\nabla_x f|^2+|\nabla_y f|^2) \dd\mu \\ &\leq&  C \, \int \, \psi(f) \, (H^{-2\eta} \, |\nabla_x f|^2+|\nabla_y f|^2) \dd\mu \, ,
\end{eqnarray*}
completing the proof.
\end{proof}

{\bf Aknowledgments.}\\ The project has benefitted from the support of ANR STAB (Stabilit\'e du comportement asymptotique d'EDP, de processus stochastiques et de leurs
discr\'etisations : 12-BS01-0019), and ANR EFI.

\bibliographystyle{plain}
\bibliography{CGMZ}

\begin{thebibliography}{10}

\bibitem{BBCG}
D.~Bakry, F.~Barthe, P.~Cattiaux, and A.~Guillin.
\newblock A simple proof of the {P}oincar\'e inequality for a large class of
  probability measures.
\newblock {\em Electronic Comm. in Probab.}, 13:60--66, 2008.

\bibitem{BCG08}
D.~Bakry, P.~Cattiaux, and A.~Guillin.
\newblock Rate of convergence for ergodic continuous {M}arkov processes:
  {L}yapunov versus {P}oincar\'e.
\newblock {\em J. Funct. Anal.}, 254(3):727--759, 2008.

\bibitem{BGL14}
D.~Bakry, I.~Gentil, and M.~Ledoux.
\newblock {\em Analysis and Geometry of {M}arkov Diffusion Operators}, volume
  348 of {\em Grundlehren der mathematischen Wissenschaften}.
\newblock Springer, 2014.

\bibitem{BCR06}
F.~Barthe, P.~Cattiaux, and C.~Roberto.
\newblock Interpolated inequalities between exponential and {G}aussian,
  {O}rlicz hypercontractivity and isoperimetry.
\newblock {\em Rev. Mat. Iberoam.}, 22(3):993--1067, 2006.

\bibitem{BL00}
S.~Bobkov and M.~Ledoux.
\newblock From {B}runn-{M}inkowski to {B}rascamp-{L}ieb and to logarithmic
  {S}obolev inequalities.
\newblock {\em Geom. Funct. Anal.}, 10:1028--1052, 2000.

\bibitem{BGM10}
F.~Bolley, A.~Guillin, and F.~Malrieu.
\newblock Trend to equilibrium and particle approximation for a weakly
  sel-fconsistent {V}lasov-{F}okker-{P}lanck equation.
\newblock {\em M2AN Math. Model. Numer. Anal.}, 44(5):867--884, 2010.

\bibitem{CG16}
P.~Cattiaux and A.~Guillin.
\newblock Hitting times, functional inequalities, {L}yapunov conditions and
  uniform ergodicity.
\newblock {\em J. Funct. Anal.}, 256:1821--1841, 2016.

\bibitem{CGWW09}
P.~Cattiaux, A.~Guillin, F.Y. Wang, and L.~Wu.
\newblock Lyapunov conditions for super {P}oincar\'e inequalities.
\newblock {\em J. Funct. Anal.}, 256(6):1821--1841, 2009.

\bibitem{CGW11}
P.~Cattiaux, A.~Guillin, and L.~Wu.
\newblock Some remarks on weighted logarithmic {S}obolev inequality.
\newblock {\em Indiana Univ. Math. J.}, 60(6):1885--1904, 2011.

\bibitem{DV01}
L.~Desvillettes and C.~Villani.
\newblock On the trend to global equilibrium in spatially inhomogeneous
  entropy-dissipating systems: the linear {F}okker-{P}lanck equation.
\newblock {\em Comm. Pure Appl. Math.}, 54(1):1--42, 2001.

\bibitem{MMS15}
J.~Dolbeault, C.~Mouhot, and C.~Schmeiser.
\newblock Hypocoercivity for linear kinetic equations conserving mass.
\newblock {\em Trans. Amer. Math. Soc.}, 367(6):3807--3828, 2015.

\bibitem{EGZ16}
A.~{Eberle}, A.~{Guillin}, and R.~{Zimmer}.
\newblock {Quantitative {H}arris type theorems for diffusions and
  {McKean-Vlasov} processes}.
\newblock Math. {A}r{X}iv:1606.06012, [math.PR], 2016.

\bibitem{EGZ17}
A.~{Eberle}, A.~{Guillin}, and R.~{Zimmer}.
\newblock {Couplings and quantitative contraction rates for {L}angevin
  dynamics}.
\newblock Math. {A}r{X}iv:1703.01617, [math.PR], 2017.

\bibitem{GW12}
A.~Guillin and F-Y. Wang.
\newblock Degenerate {F}okker-{P}lanck equations: {B}ismut formula, gradient
  estimate and {H}arnack inequality.
\newblock {\em J. Differential Equations}, 253(1):20--40, 2012.

\bibitem{HN04}
F.~H{\'e}rau and F.~Nier.
\newblock Isotropic hypoellipticity and trend to equilibrium for the
  {F}okker-{P}lanck equation with a high-degree potential.
\newblock {\em Arch. Ration. Mech. Anal.}, 171(2):151--218, 2004.

\bibitem{Mon15b}
P.~Monmarch\'e.
\newblock {Generalized $\Gamma$ calculus and application to interacting
  particles on a graph}.
\newblock Math. {A}r{X}iv:1510.05936v3 [math.PR], 2015.

\bibitem{Tal02}
D.~Talay.
\newblock Stochastic {H}amiltonian systems: exponential convergence to the
  invariant measure, and discretization by the implicit {E}uler scheme.
\newblock {\em Markov Process. Related Fields}, 8(2):163--198, 2002.
\newblock Inhomogeneous random systems.

\bibitem{Villani}
C.~Villani.
\newblock Hypocoercivity.
\newblock {\em Mem. Amer. Math. Soc.}, 202(950):iv+141, 2009.

\bibitem{Wbook}
F.~Y. Wang.
\newblock {\em Functional inequalities, {M}arkov processes and {S}pectral
  theory}.
\newblock Science Press, Beijing, 2005.

\bibitem{Wu01}
L.~Wu.
\newblock Large and moderate deviations and exponential convergence for
  stochastic damping {H}amiltonian systems.
\newblock {\em Stochastic Process. Appl.}, 91(2):205--238, 2001.

\end{thebibliography}

\end{document}